\documentclass{ifacconf}

\usepackage[utf8]{inputenc}

\usepackage{graphicx}      
\usepackage{natbib}        
\usepackage{mathrsfs}
\usepackage{amsmath} %
\usepackage{mathtools} %
\usepackage{amsfonts} %
\usepackage{amssymb} %
\usepackage{comment} %
\usepackage{algorithm} %
\usepackage{algpseudocode} %
\usepackage[nodisplayskipstretch]{setspace} %
\usepackage{url} %
\usepackage[draft]{todonotes}

\usepackage{cite}

\newcommand{\pnote}[1]{\textcolor{purple}{\small {\textbf{(Pepe:} \textbf{) }}}}
\newcommand{\cnote}[1]{\textcolor{red}{\small {\textbf{(Carmen:} \textbf{) }}}}
\newcommand{\dnote}[1]{\textcolor{blue}{\small {\textbf{(Dim:} \textbf{) }}}}

\usepackage{graphics} 
\usepackage{epsfig} 
\usepackage{times} 

\usepackage{appendix}

\newtheorem{theorem}{Theorem}[section]

\newtheorem{lemma}{Lemma}[section]

\makeatletter
\DeclareRobustCommand{\qed}{%
  \ifmmode 
  \else \leavevmode\unskip\penalty9999 \hbox{}\nobreak\hfill
  \fi
  \quad\hbox{\qedsymbol}}
\newcommand{\openbox}{\leavevmode
  \hbox to.77778em{%
  \hfil\vrule
  \vbox to.675em{\hrule width.6em\vfil\hrule}%
  \vrule\hfil}}
\newcommand{\qedsymbol}{\openbox}
\newenvironment{proof}[1][\proofname]{\par
  \normalfont
  \topsep6\p@\@plus6\p@ \trivlist
  \item[\hskip\labelsep\itshape
    #1.]\ignorespaces
}{%
  \qed\endtrivlist
}
\newcommand{\proofname}{Proof}
\makeatother

\theoremstyle{definition}

\newtheorem{remark}{Remark}

\newcommand{\lemref}[1]{Lemma \ref{#1}}

\newcommand{\figref}[1]{Figure \ref{#1}}

\newcommand{\bb}[1]{\mathbf{#1}}

\newcommand{\Vmaxin}{\bar{\mathcal{V}}^{in}}
\newcommand{\Vminin}{\underline{\mathcal{V}}^{in}}
\newcommand{\Vmaxout}{\bar{\mathcal{V}}^{out}}
\newcommand{\Vminout}{\underline{\mathcal{V}}^{out}}
\newcommand{\Vin}{\mathcal{V}^{in}}

\newcommand{\Dout}{\mathcal{D}^{out}}
\newcommand{\Vout}{\mathcal{V}^{out}}
\newcommand{\PhiX}{\mathbf{\Phi}_x}
\newcommand{\PhiU}{\mathbf{\Phi}_u}
\newcommand{\PhiXtk}[2]{{\Phi}^{#1,#2}_x}

\newcommand{\Phitkij}[4]{\left[{\Phi}^{#1,#2}\right]_{#3,#4}}
\newcommand{\PhiXtkij}[4]{\left[{\Phi}^{#1,#2}_x\right]_{#3,#4}}
\newcommand{\PhiUtk}[2]{{\Phi}^{#1,#2}_u}
\newcommand{\PhiUtkij}[4]{\left[{\Phi}^{#1,#2}_u\right]_{#3,#4}}

\newcommand{\Deltatkij}[4]{\left[{\Delta}^{#1,#2}\right]_{#3,#4}}
\newcommand{\SetS}{\mathcal{S}}

\begin{document}
\begin{frontmatter}

\title{Distributed Linear Quadratic Regulator Robust to Communication Dropouts\thanksref{footnoteinfo}} 

\thanks[footnoteinfo]{This work was partially supported by the European Research Council (Advanced Research Grant $769051$-OCONTSOLAR) and the MINECO-Spain project DPI2017-86918-R.}

\author[First]{C. Amo Alonso} 
\author[First]{D. Ho} 
\author[Second]{J.M. Maestre} 

\address[First]{Computing and Mathematical Sciences, California Institute of Technology, 
   Pasadena, CA, USA (e-mail: camoalon@caltech.edu, dho@caltech.edu).}
\address[Second]{Dep. of Systems and Automation Engineering, University of Seville, Spain (e-mail: pepemaestre@us.es)}

\begin{abstract}                
We present a solution to deal with information package dropouts in distributed controllers for large-scale networks. We do this by leveraging the System Level Synthesis approach, a control framework particularly suitable for large-scale networks that addresses information exchange in a very transparent manner. To this end, we propose two different schemes for controller synthesis and implementation. The first one synthesizes a controller inherently robust to dropouts, which is later implemented in an offline fashion. For the second approach, we synthesize a collection of controllers offline and then switch between different controllers online depending on the current dropouts detected in the system. The two approaches are illustrated and compared by means of a simulation example.

\end{abstract}


\end{frontmatter}

\section{Introduction}
\vspace{-3mm}

In distributed control methods, local controllers coordinate their actions to deal with mutual interaction issues and improve overall system performance. The information exchanged is typically performed using a communication network, which is a potential source of vulnerability due to well-known problems as packet losses and cyber-aggressions. See for example~\citep{lun2019state} for a comprehensive review of cyber-physical systems security and~\citep{sandberg2015cyberphysical}, where attack and defense strategies are shown for Network Control Systems. 

Unreliable networks where packet dropouts can occur have been explored in the past through many different control methods. For instance, in \citep{quevedo2015stochastic,mishra2018stabilizing} the features of stochastic model predictive control are exploited to buffer the input sequence and mitigate communication losses, and in~\citep{cetinkaya2015event}, where time-inhomogeneous Markov chains are used to model random packet losses in a control scheme where a feedback gain is used. 

In this paper, we focus on a version of the linear quadratic regulator (LQG) problem that incorporates communication constraints as well as probabilistic communication dropouts between subcontrollers. To do this, we study this problem in the context of the recently proposed System Level Synthesis (SLS) framework~\citep{anderson_system_2019, wang_separable_2017, matni_scalable_2017}, which allows for the synthesis and implementation of distributed and localized controllers in a scalable way, making it a very suitable framework for large-scale networks. In particular, the SLS framework explicitly accounts for the information structure of the network, so the sub-controllers are sparsely connected, i.e., each of the sub-controllers only exchanges information with other sub-controllers within its communication range. Besides, it allows to localize the effects of the disturbances within that range, which ultimately allows the subcontrollers to only access local information and solve for local subproblems of much smaller complexity than the complexity of the whole problem. Given this ability to localize, together with the explicit dependence on the information structure, it is important to analyze how an SLS approach can be affected by packet losses, and also to propose means to relieve this issue. 

In this context, our main contribution is an algorithm for synthesis and implementation of SLS controllers that allows for a time-varying communication topology to deal with packet losses. To this end, two different approaches are proposed:
\begin{itemize}
    \item Offline distributed controller synthesis and implementation of SLS controllers intrinsically robust to communication dropouts. This requires minimal online computations.

    \item Offline synthesis of SLS controllers together with an online implementation strategy that adapts the controllers to different communication topologies based on the sensed package losses.
    In contrast to the first approach, this procedure requires more online computations, but comes with increased control performance as it can adapt to changing communication topologies.
\end{itemize} 

Since SLS is a natural framework to perform distributed control and the communication structure appears explicitly in the formulation, the presented derivations are carried out by leveraging the SLS framework. All the results are suitably distributed among the sub-controllers of the network. Finally, a simulation example is given to illustrate the proposed methods. 

The remainder of this paper is organized as follows. In Section 2, we present the problem formulation. Section 3 introduces the offline distributed synthesis and implementation strategies using the robust SLS framework, and we provide a result for the distribution of the robust SLS synthesis. Section 4 presents an algorithm for offline distributed synthesis of SLS robust controllers together with an online distributed implementation that allows to tackle the communication dropouts. In Section 5, we illustrate the proposed approach via simulation. We end in Section 6 with conclusion, and directions for future work. 

\textbf{Notation.} Lower-case and upper-case Latin and Greek letters such as $x$ and $A$ denote vectors and matrices respectively, although lower-case letters might also be used for scalars or functions (the distinction will be apparent from the context). We use bracketed indices to denote the time of the true system, i.e., the system is at state $x(t)$ at time $t$. Subscripts denote time indices within a loop, i.e., $x_t$ denotes the $t^\mathrm{th}$ state within the loop. To denote subsystem variables, we use square bracket notation, i.e. $[x]_{i}$ denotes the components of vector $x$ that correspond to subsystem $i$. Boldface lower and upper case letters such as $\mathbf{x}$ and $\mathbf{K}$ denote finite horizon signals and lower block triangular operators, respectively:
\begin{equation*} 
\mathbf{x}=\left[\begin{array}{c} x_{0}\\x_{1}\\\vdots\end{array}\right],\
    \mathbf{K}=\left[\begin{array}{cccc}K^{0,1} & 0 & & \\ K^{0,2} & K^{1,1} & & \\ \vdots & \ddots & \ddots  \end{array}\right],
\end{equation*}
where each $K^{i,j}$ is a matrix of compatible dimension. In this notation, $\mathbf{K}$ is the matrix representation of the convolution operation induced by a time varying controller $K_t(x_{0:t})$, so that $u_{t}=\mathbf{K}^{t,:}\mathbf{x}$, where $\mathbf{K}^{t,:}$ represents the $t$-th block-row of $\mathbf{K}$. 

\section{Problem Statement}
\vspace{-3mm}
Consider a discrete-time linear time invariant (LTI) dynamical system with dynamics:
\begin{equation}
    x(t+1)=Ax(t)+Bu(t)+w(t),
    \label{eq: dynamics}
\end{equation}
where $x(t)\in\mathbb{R}^{n}$ is the state, $u(t)\in\mathbb{R}^{p}$ is the control input, and $w(t)\in\mathbb{R}^{n}$ is an exogenous disturbance, which we assume to be additive-white Gaussian noise. The system is structured and can be described as a collection of $N$ interconnected subsystems, with local state, control, and disturbance inputs 
given by $[x]_i$, $[u]_i$, and $[w]_i$ for each subsystem $i$. Accordingly if $\mathcal{S}$ is a set of subsystems, we will understand $[x]_{\mathcal{S}}$, $[u]_{\mathcal{S}}$ to be the concatenation of $[x]_i$, $[u]_i$ for all $i\in \mathcal{S}$. The matrices $A$ and $B$ can also be partitioned into a compatible local block structure $[A]_{ij}$, $[B]_{ij}$ so the local dynamics for subsystem $i$ are:
\begin{equation*}
    [x]_{i}(t+1)=\sum_{j}[A]_{ij}[x]_{j}(t)+\sum_{j}[B]_{ij}[u]_{j}(t)+[w]_{i}(t).
\end{equation*}
The interconnection topology of the system can be modeled as a time-invariant unweighted directed graph $\mathcal{G}(E,V)$, where each subsystem $i$ is identified with a vertex $v_{i}\in V$ and an edge $(v_i,v_j)\in E$ exists whenever $[A]_{ij}\neq 0$ or $[B]_{ij}\neq 0$.

Inspired by control applications in the large-scale system setting, we will assume that subcontroller $[u]_i$ is only able to directly measure its own subsystem state $[x]_i$, but can receive information with no delay from a small set of neighboring sub-controllers $\Vmaxin(i)\in V$ through some kind of communication network and can send information with no delay to a small set of neighboring sub-controllers $\Vmaxout(i)\in V$ in the same manner. To model this interaction more precisely, assume that every controller $i$ has an internal state $[p_t]_i$ and at every time-step it performs the following two operations in order:
\begin{enumerate}\label{const:com}
    \item compute message $[m_t]_i$ based on history of internal state and own measurement $[x_t]_i$ 
    \begin{align} \label{eq:step1}
    [m_t]_{i} = [h_t]_i([x_t]_i, [p_t]_{i},[p_{t-1}]_{i},\dots), 
    \end{align}
    \item exchange messages with neighbors and update internal state:
    \begin{align} \label{eq:step2}
    [p_t]_{i} = [m_t]_{\Vmaxin(i)},
    \end{align}    
    \item compute control action based on own measurement $[x_t]_i$ and history of internal state $[p_t]_{i}$:
    \begin{align} \label{eq:step3}
    [u_t]_{i} = [k_t]_i([x_t]_i, [p_t]_{i},[p_{t-1}]_{i},\dots).
    \end{align}
\end{enumerate}

In real applications communication networks rarely offer a continuously reliable connection between subsystems. Thus, we will assume that some communication links in the network can occasionally drop packets. We model this by defining the constraint that subsystem $i$ at time $t$ can receive information only from the subset $\Vin_t(i) \subset \Vmaxin(i)$ and can only send information to subset $\Vout_t(i) \subset \Vmaxout(i)$. 
We define $\Dout_t(i) := \Vmaxout(i)/ \Vout_t(i)$ as the subset of systems which dropped packages from system $i$, and assume that for each $i$ there is a small subset of neighbors  $\Vminin(i)\subset\Vin_t(i) $ and $\Vminout(i)\subset\Vout_t(i) $ that guarantee communication without dropouts.

If we consider dropouts, then the controller implementation gets perturbed by replacing update step \eqref{eq:step2} with 
\begin{equation*} 
    [p_t]_{i} = [m_t]_{\Vin_t(i)},
    \tag{3'}
\end{equation*}   
which simply states that the internal state update step does not get all messages of its neighbors.

We will model the dropouts occurring for every subsystem $i$ as a random process, in particular we will assume $\Vin_t(i)$ to be independent random variables over $t$, that are identically distributed according to some probability mass function 
\begin{align}
\label{eq:pmf}
f_i: 2^{\Vmaxin(i) / \Vminin(i)} \rightarrow [0,1].
\end{align}
\begin{remark}
$2^{\Vmaxin(i) / \Vminin(i)}$  means the power set of all possible subsystems $\Vmaxin(i) / \Vminin(i)$ that could drop communication to subsystem $i$.
\end{remark}
Furthermore, assume the random sets $\Vin_t(i)$ are also independent for different $i$. We will refer to the joint-probability mass function as $f$ and we will abbreviate $$\Vin_t : =\left\{\Vin_t(1),\Vin_t(2),\dots,\Vin_t(N) \right\}$$
as the collection of subsets of $\{\Vin_t(i)\}$ at time $t$.\\
\vspace{1mm}
With these definitions at hand, we can formulate our goals as to minimize the expected average quadratic cost over the gaussian disturbance $w$ and the distribution of the dropouts occurring, i.e.,
\begin{equation}\label{eq: prob st}
    \begin{aligned}
    &\underset{\{h_t,k_t\}}{\text{min}} & &\mathbb{E}_{\Vin_t \sim f}\underset{T\rightarrow\infty}{\text{lim}}\frac{1}{T}\sum_{t=0}^{T} \mathbb{E}[x_t^{\mathsf{T}}Qx_t+u_t^{\mathsf{T}}Ru_t]\\
    & \mathrm{s.t.} &  &\begin{aligned}
              &x_{t+1}=Ax_t+Bu_t+w_t\ \ \forall t=0,1,...,\\
              & \eqref{eq:step1}, \eqref{eq:step3},(3'),
              \end{aligned}
    \end{aligned}
\end{equation}

where $\mathbb{E}$ is the mathematical expectation and $(Q,R)$ are positive definite cost matrices.

In this work, we propose two ways to tackle this problem, both formulated in the SLS framework since it provides a tractable way to deal with the communication constraint \eqref{eq:step2}. The first strategy is to the tackle problem with an offline controller. To do this, we will synthesize a controller offline inherently robust to a specified set of communication dropout patterns. This will penalize performance but will guarantee stability, as we will show in section 4. We also present a different approach where different controllers for different information exchange topologies are synthesized offline, which then are implemented by an online controller that is senses dropouts and changes the communication strategy while guaranteeing stability. Performance of these two strategies is compared through simulation.

\section{Preliminaries: System Level Synthesis}
\vspace{-3mm}

In this section we present an abridged version of the SLS framework (see \citep{anderson_system_2019,matni_scalable_2017,wang_system_nodate} and references therein). The SLS framework will be used in the derivations presented in this paper due to its ability to impose locality constraints, as well as its ability to distribute both controller synthesis and implementation. Recent work has extended the SLS approach to the time-varying system \citep{anderson_system_2019,ACCho2019} and even general nonlinear systems \citep{ho2020system}. 

\subsection{Time domain System Level Synthesis}
\vspace{-3mm}
Consider the system with dynamics \eqref{eq: dynamics}. Let $u_t$ be a causal time-varying state-feedback control law, i.e., $u_t=K_t(x_0,x_1,...,x_t)$ where $K_t$ is some linear map. We will denote $\bb{K}$ the general bounded causal linear operator mapping state sequences $\bb{x}$ to input sequences $\bb{u} = \bb{K}\bb{x}$. Although rather informal, for ease of exposition we will think of linear bounded operators as infinite dimensional lower triangular matrices. To this end, define $\bb{A}:=\mathrm{blkdiag}(A,A,\dots)$, $\bb{B}:=\mathrm{blkdiag}(B,B,\dots)$ as infinite dimensional block-diagonal matrices with $A$ and $B$ on the diagonal and let $Z$ be the block-downshift operator\footnote{A matrix with identity matrices along its first block sub-diagonal and zeros elsewhere.}. Using this notation, the dynamics \eqref{eq: dynamics} impose the following relationship between the state $\bb{x}$, input $\bb{u}$ and disturbance $\bb{w}$:
\begin{equation*} 
\begin{split}
\mathbf{x} = Z(\bb{A}+\bb{B}\mathbf{K})\mathbf{x+\mathbf{w}}.
\end{split}
\end{equation*}
Overall, the closed loop behavior of the system in (\ref{eq: dynamics}) under the feedback law $\mathbf{K}$ can be entirely characterized by
\begin{subequations}
\begin{align}
\mathbf{x} & = (I-Z(\bb{A}+\bb{B}\mathbf{K}))^{-1}\mathbf{w}=:\mathbf{\Phi}_{x}\mathbf{w}\\
\mathbf{u} & = \mathbf{K}(I-Z(\bb{A}+\bb{B}\mathbf{K}))^{-1}\mathbf{w}:\mathbf{\Phi}_{u}\mathbf{w},
\end{align}
\end{subequations} \label{eq: closed loop A&B}
where the operators $\mathbf{\Phi}_{x}$ and $\mathbf{\Phi}_{u}$ are the closed loop maps from the disturbance $\mathbf{w}$ to the state $\mathbf{x}$ and control input $\mathbf{u}$, respectively.

The SLS framework relies of this parametrization for the optimal controller synthesis to be performed by directly optimizing over the system responses $\mathbf{\Phi}_{x}$ and $\mathbf{\Phi}_{u}$, as opposed to the controller map $\mathbf{K}$ itself. 
\\
\begin{theorem}{\label{thm: SLS}}
For the dynamics (\ref{eq: dynamics}) state feedback law $\mathbf{K}$ defining the control action as $\mathbf{u} = \mathbf{K}\mathbf{x}$, the following are true
\begin{enumerate}
    \item the affine subspace defined by 
    \begin{equation}\label{eq: constraint}
        \left[I-ZA\ \ -ZB\right]\left[\begin{array}{c}\mathbf{\Phi}_{x}\\\mathbf{\Phi}_{u}\end{array}\right] = I
    \end{equation}
    parameterizes all possible system responses.
    
    \item for any transfer matrices $\left\{\mathbf{\Phi}_{x},\mathbf{\Phi}_{u}\right\}$ satisfying (\ref{eq: constraint}), the controller $\mathbf{K} = \mathbf{\Phi}_{u}\mathbf{\Phi}_{x}^{-1}$ achieves the desired response.
\end{enumerate}
\end{theorem}

Hence, by using the SLS framework optimal control problems can be reformulated into convex optimization problems. A detailed description on how to do this is provided in \citep{anderson_system_2019}. One of the main advantages of formulating an optimal control problem with the system response parametrization is that it allows to impose \eqref{eq:step1}, \eqref{eq:step2} and \eqref{eq:step3} in a convex manner by imposing the system responses to be localized. As shown in \citep{anderson_system_2019}, communication constraints on controllers of the form $\bb{K} = \PhiU\PhiX^{-1}$ can be easily incorporated into optimal control problems, by requiring that the closed loop maps $\left\{\mathbf{\Phi}_{x},\mathbf{\Phi}_{u}\right\}$ lie in a suitably chosen subspace $\mathcal{S}$. Overall, many traditionally non-convex optimal control problems can be recast into their equivalent convex SLS representation of the form
\begin{equation}
\begin{aligned}\label{eq:slso}
& \underset{\mathbf{\Phi}_{x},\mathbf{\Phi}_{u}}{\text{min}} &  &g(\mathbf{\Phi}_{x},\mathbf{\Phi}_{u})\\
& \ \text{s.t.} &  &\begin{aligned} 
    &\left[I-ZA\ \ -ZB\right]\left[\begin{array}{c}\mathbf{\Phi}_{x}\\\mathbf{\Phi}_{u}\end{array}\right] = I \\
    &\left\{\mathbf{\Phi}_{x},\mathbf{\Phi}_{u}\right\}\in\mathcal{S},
\end{aligned}
\end{aligned}
\end{equation}
where the optimization is phrased over the feasible closed loop maps $\left\{\mathbf{\Phi}_{x},\mathbf{\Phi}_{u}\right\}$. Moreover, $\mathcal{S}$ can incorporate other convex constraints imposed on the system responses, i.e., finite impulse response (FIR), performance bounds, etc. 

\subsection{Virtually localizable System Level Synthesis}
\vspace{-3mm}
In the previous subsection the SLS framework was presented. Oftentimes the constraint $\{\PhiX,\PhiU\} \in \mathcal{S}$ can be too restrictive to impose on the closed loop maps, while it is only used for certifying that certain communication constraints on the controller $\mathbf{K} = \PhiU\PhiX^{-1}$ are enforced. A robust-variant of the SLS parametrization described above was introduced by \citep{matni_scalable_2017}, and later generalized in \citep{ACCho2019} for the general setting of time-varying systems, which addresses this issue. The following general approach can be found in \citep{anderson_system_2019} and is based on the following result:
\\
\begin{theorem}\label{thm:robstab}
Let $(\mathbf{\tilde{\Phi}}_{x}, \mathbf{\tilde{\Phi}}_{u}, \mathbf{\Delta})$ be a solution to
    \begin{equation*}
        \left[I-ZA\ \ -ZB\right]\left[\begin{array}{c}\mathbf{\tilde{\Phi}}_{x}\\\mathbf{\tilde{\Phi}}_{u}\end{array}\right] = (I+\mathbf{\Delta}).
    \end{equation*}
    
Then, the controller implementation 
\begin{equation}{\label{eq: robust implementation}}
\begin{aligned}
        \mathbf{u}=\mathbf{\tilde{\Phi}}_{u}\mathbf{\hat{w}},\ \ \  \mathbf{\hat{x}}=(I - \mathbf{\tilde{\Phi}}_{x})\mathbf{\hat{w}},\ \ \ 
        \mathbf{\hat{w}}=\mathbf{x}-\mathbf{\hat{x}},
\end{aligned}
\end{equation}

internally stabilizes the system \eqref{eq: dynamics} if and only if $(I+\mathbf{\Delta)^{-1}}$ is stable. Furthermore, the actual system responses achieved are given by:
\begin{equation}\label{eq: system response}
    \left[\begin{array}{c}\mathbf{x}\\\mathbf{u}\end{array}\right] = \left[\begin{array}{c}\tilde{\mathbf{\Phi}}_{x}\\\tilde{\mathbf{\Phi}}_{u}\end{array}\right](I+\mathbf{\Delta})^{-1} \mathbf{w}.
\end{equation}

\label{thm: robust SLS}
\end{theorem}{}

This theorem allows for a reformulation of problem \eqref{eq:slso}, which allows to search over controller implementations $\bb{K} = \tilde{\mathbf{\Phi}}_{u}\tilde{\mathbf{\Phi}}_{x}^{-1}$ where $\mathbf{\tilde{\Phi}}_{x}, \mathbf{\tilde{\Phi}}_{u}$ do not have to be necessarily closed loop maps. As stated in \citep{anderson_system_2019}, a sufficient condition for $(I-\mathbf{\Delta})^{-1}$ to be stable is that any of the induced norms $\left\Vert\mathbf{\Delta}\right\Vert_{\mathcal{H}_{\infty}},\left\Vert\mathbf{\Delta}\right\Vert_{\mathcal{L}_{1}},\left\Vert\mathbf{\Delta}\right\Vert_{\varepsilon_{1}}$ are smaller than $<1$. In particular, the optimal control problem \eqref{eq:slso} can be relaxed to:
\\
\begin{equation}\label{eq: robust SLS}
\begin{aligned}
& \underset{\tilde{\mathbf{\Phi}}_{x}, \tilde{\mathbf{\Phi}}_{u}, \mathbf{\Delta}}{\text{min}} &  &g(\left[\begin{array}{c}\tilde{\mathbf{\Phi}}_{x}\\\tilde{\mathbf{\Phi}}_{u}\end{array}\right](I+\mathbf{\Delta})^{-1})\\
& \ \text{s.t.} &  &\begin{aligned} 
    &\left[I-ZA\ \ -ZB\right]\left[\begin{array}{c}\tilde{\mathbf{\Phi}}_{x}\\\tilde{\mathbf{\Phi}}_{u}\end{array}\right] = I+\mathbf{\Delta} \\
    &\left\{\tilde{\mathbf{\Phi}}_{x},\tilde{\mathbf{\Phi}}_{u}\right\}\in\mathcal{L}_d\cup\mathcal{S},\ \ \left\Vert{\mathbf{\Delta}}\right\Vert_{*}<1,
\end{aligned}
\end{aligned}
\end{equation}

where $*$ in $\left\Vert\cdot\right\Vert_{*}$ is one of $\mathcal{H}_{\infty}$, $\mathcal{L}_{1}$, $\varepsilon_{1}$.  
\section{Controller Structure and Distributed Synthesis}
\vspace{-3mm}
Here we introduce the reformulation of problem \eqref{eq: prob st} into the SLS framework, and illustrate the controller structure and its implementation and the communication model. This will provide us with a set of tools that will be useful in the next sections.


In what follows, we use the notation $\PhiXtk{t}{k}$, $\PhiUtk{t}{k}$ and $\PhiXtkij{t}{k}{i}{j}$, $\PhiUtkij{t}{k}{i}{j}$ to index sub-matrices of $\PhiX$ and $\PhiU$
\begin{align}\label{eq:notat1}
   &\PhiX = \begin{bmatrix}\PhiXtk{0}{1}&0&0&\dots \\
   \PhiXtk{0}{2}&\PhiXtk{1}{1}&0&\dots\\
   \PhiXtk{0}{3}&\PhiXtk{1}{2}&\PhiXtk{2}{1}&\dots\\
    \vdots&\vdots&\vdots&\ddots
   \end{bmatrix}
\end{align}
\begin{align}\label{eq:notat2}
   &\PhiXtk{t}{k}= \begin{bmatrix}\PhiXtkij{t}{k}{1}{1}&\PhiXtkij{t}{k}{1}{2}&\dots&\PhiXtkij{t}{k}{1}{n} \\
   \PhiXtkij{t}{k}{2}{1}&\PhiXtkij{t}{k}{2}{2}&\dots&\PhiXtkij{t}{k}{2}{n} \\
    \vdots&\vdots&\ddots&\vdots\\
    \PhiXtkij{t}{k}{n}{1}&\PhiXtkij{t}{k}{n}{2}&\dots&\PhiXtkij{t}{k}{n}{n}
   \end{bmatrix}
\end{align}
and the same convention is understood for $\PhiU$ and $\bb{\Delta}$. Notice that we are abusing notation, and although we are referring to the system responses in \eqref{eq: system response}, we will drop the tilde for convenience. Also, we will use the abbreviation for the columns
\begin{align}\label{eq:notat3}
    &[\phi^{t}_x]_{i} := \PhiUtkij{t}{:}{:}{i} &[\phi^{t}_u]_{i} := \PhiXtkij{t}{:}{:}{i}.
\end{align}
With this notation, we can write the controller implementation \eqref{eq: robust implementation} into its distributed form:
\begin{subequations}
\label{eq:implementation}
\begin{align}
    [\hat{w}_t]_j &= [x_t]_j-\sum\limits_{i}\sum\limits^{t+1}_{k=2}\PhiXtkij{t+1-k}{k}{j}{i}[\hat{w}_{t+1-k}]_i \\
    [u_t]_j &= \sum\limits_{i}\sum\limits^{t+1}_{k=1}\PhiUtkij{t+1-k}{k}{j}{i}[\hat{w}_{t+1-k}]_i. 
\end{align}
\end{subequations}
The above controller implementation can be put in the form of the communication model described in \eqref{eq:step1}, with the following correspondence \eqref{eq:step2} in the SLS notation:
\begin{align}
    [m_t]_i &= \{[\phi^{t}_x]_{i}[\hat{w}_{t}]_i,[\phi^{t}_u]_{i}[\hat{w}_{t}]_i\}\label{eq:step1SLS} \\
    [p_t]_i &= \{[\phi^{t}_x]_{j}[\hat{w}_{t}]_j,[\phi^{t}_u]_{j}[\hat{w}_{t}]_j\}^{N}_{j=1}\label{eq:step2SLS}.
\end{align}
\begin{remark}
The above communication scheme expresses that every subsystem $i$ decides on the columns $[\phi^{t}_x]_{i}$, $[\phi^{t}_u]_{i}$ and it shares it with all other subsystems and it can be verified that this is consistent with the original controller equations \eqref{eq:implementation}.
\end{remark}
Recalling our original problem \eqref{eq: prob st}, notice that we can equivalently formulate the dropout constraint 
$$[p_t]_{i} = [m_t]_{\Vin_t(i)}$$ as the following sparsity constraint on $[\phi^{t}_x]_{j}$ and $[\phi^{t}_u]_{j}$:
\begin{subequations}
\label{eq:sparsity}
\begin{align*}
    [[\phi^{t}_x]_{i}]_{\Dout_t(i)} = 0 &\Leftrightarrow [\phi^{t}_x]_{i} = [[\phi^{t}_x]_{i}]_{\Vout_t(i)},\\
    [[\phi^{t}_u]_{i}]_{\Dout_t(i)} = 0 &\Leftrightarrow [\phi^{t}_u]_{i} = [[\phi^{t}_u]_{i}]_{\Vout_t(i)},
\end{align*}
\end{subequations}
where $[[\phi^{t}_x]_{j}]_{\mathcal{S}}$ stands for 
\begin{equation*}
    [[\phi^{t}_x]_{i}]_{\mathcal{S}} := \PhiXtkij{t}{:}{\mathcal{S}}{i},\quad [[\phi^{t}_u]_{i}]_{\mathcal{S}} := \PhiUtkij{t}{:}{\mathcal{S}}{i}.
\end{equation*}

With this in mind, together with equation \eqref{eq: system response}, we can reformulate optimization \eqref{eq: prob st} into a robust SLS problem\footnote{Notice that, as shown in \citep{wang_localized_2014}, the localized LQR cost in \eqref{eq: prob st} in the presence of additive-white Gaussian noise can be reformulated into the $\mathcal{H}_2$ norm of the weighted transfer matrices $\mathbf{{\Phi}}_x$ and $\mathbf{{\Phi}}_u$} as: 
\begin{equation}\label{eq: prob st SLS}
    \begin{aligned}
    &\underset{\mathbf{\Phi}_{x}, \mathbf{\Phi}_{u}, \mathbf{\Delta}}{\text{min}} & &\mathbb{E}_{\Vin_t \sim f}\left\Vert\begin{bmatrix}\mathbf{Q} & \mathbf{R}\end{bmatrix} \begin{bmatrix}\mathbf{\Phi}_{x}\\\mathbf{\Phi}_{u}\end{bmatrix}(I+\mathbf{\Delta})^{-1}\right\Vert_{\mathcal{H}_2}\\
    & \text{s.t.} &  &\begin{aligned}
              &Z^T_{AB}\left[\begin{array}{c}\mathbf{\Phi}_{x}\\\mathbf{\Phi}_{u}\end{array}\right] = \mathbf{\Delta},\ \ \left\Vert{\mathbf{\Delta}}\right\Vert_{*}<1,\\
              &[\phi^{t}_x]_{i} = [[\phi^{t}_x]_{i}]_{\Vout_t(i)},\ [\phi^{t}_u]_{i} = [[\phi^{t}_u]_{i}]_{\Vout_t(i)}\\
              & \forall\ t=0,1,\ldots
              \end{aligned}
    \end{aligned}
\end{equation}
where $Z^T_{AB}=\left[Z^T-A,\ \ -B\right]$, $\bb{Q}:=\mathrm{blkdiag}(Q^{\frac{1}{2}},Q^{\frac{1}{2}},\dots)$, $\bb{R}:=\mathrm{blkdiag}(R^{\frac{1}{2}},R^{\frac{1}{2}},\dots)$, and $*$ in $\left\Vert\cdot\right\Vert_{*}$ is one of $\mathcal{H}_{\infty}$, $\mathcal{L}_{1}$, $\varepsilon_{1}$. Notice that the last constraint is using the notation $\Vout_k(i)$, which accounts for the presence of dropouts for each $k<t\ \forall t$.
\vspace{-1mm}
\subsection{Relaxation of $\mathcal{H}_2$ Optimal Robust controller synthesis}
\vspace{-2mm}
Equation \eqref{eq: prob st SLS} presents a reformulation of the original problem \eqref{eq: prob st} into one featuring the SLS formulation. This reformulation will allow to make this problem tractable given the transparency with which communication structure constraints can be imposed in SLS. However, one of the major limitations of \eqref{eq: prob st SLS} is that the objective function is non-convex in $\mathbf{\Delta}$. Not only that, but it is not obvious how to perform the computation distributedly among subcontrollers. In what follows we present a way to solve a relaxation of \eqref{eq: prob st SLS} by means of a quasi-distributed convex optimization. In the following derivation we choose the $\mathcal{L}_1$ norm as the robustness criterion for $\left\Vert\cdot\right\Vert_{*}$ and give an appropriate bound on the $\mathcal{H}_2$ norm.

We start by recalling a result on Neumann series of bounded linear operators from \lemref{lem:Neuma} (Appendix A). It can be leveraged to obtain the following $\mathcal{H}_2$ inequality for linear time-invariant operators in terms of the $\| . \|_{2\leftarrow 1}$-norm as in \eqref{lem:G21}.
\\
\begin{lemma}\label{lem:G21}
Let $\bb{\Phi}$ be a time-invariant linear operator with $\Phitkij{t}{k}{i}{j}$ and $[\phi^{t}]_{i}$ referring to the decompositions in \eqref{eq:notat1}, \eqref{eq:notat2}, \eqref{eq:notat3}. If $\|\bb{\Phi}\|_{\mathcal{H}_2} < \infty$, then 
\begin{align}
    \| \bb{\Phi}\|_{2 \leftarrow 1} &: = \sup\limits_{\|\bb{w}\|_1 \leq 1} \left\|\bb{\Phi}\bb{w}\right\|_2 = \max_{i,t} \sqrt{\|[\phi^{t}]_{i}\|^2_{F}}.
    \label{eq: G21}
\end{align}
\end{lemma}
\vspace{-6mm}
\begin{proof}
See Appendix A.
\end{proof}
\begin{lemma}\label{lem:robust relax}
Assume $\bb{G}$, $\bb{\Delta}$ are linear time-invariant operators of dimension $n$ with $\|\bb{\Delta}\|_{\mathcal{L}_1} \leq \lambda<1$ and $\|\bb{G} \|_{\mathcal{H}_2} < \infty$. Then the following bound holds:
\begin{align*}
    \left\|\bb{G}(\bb{I}-\bb{\Delta})^{-1} \right\|_{\mathcal{H}_2} \leq \|\bb{G}\|_{2 \leftarrow 1}\frac{n}{1-\lambda}.
\end{align*}
\end{lemma}
\begin{proof}
    Using \lemref{lem:Neuma} we have $\bb{G}(\bb{I}-\bb{\Delta})^{-1} = \sum\limits^{\infty}_{k=0} \bb{G}\bb{\Delta}^{k},$ where $\bb{\Delta}^{k} = \underbrace{\bb{\Delta} \circ \bb{\Delta} \circ \dots \circ \bb{\Delta} }_{k}$. First consider the $\mathcal{H}_2$ norm of the terms $\bb{G}\bb{\Delta}^{k}$. We can write
        $\|\bb{G}\bb{\Delta}^{k}\|^{2}_{\mathcal{H}_2} = \sum\limits^{n}_{i=1}  \|\bb{G}\bb{\Delta}^{k} \bb{s}^{(i)}\|^{2}_2$
    where $\bb{s}^{(i)}$ is the sequence where $s^{(i)}_0 = e_i $ \footnote{$e_i$ is the unit vector in the $i^{th}$ entry.} and $s^{(i)}_k = 0$ for all $k>0$. Since, $\|\bb{s}^{(i)}\|_1 =1$ and $\|\bb{\Delta}^{k}\|_{\mathcal{L}_1} \leq \|\bb{\Delta}\|^{k}_{\mathcal{L}_1} \leq \lambda^k$, we have $\|\bb{\Delta}^{k} \bb{s}^{(i)}\|_{1} \leq \lambda^k$. This gives us the bound:
    \begin{align*}
        \|\bb{G}\bb{\Delta}^{k}\|_{\mathcal{H}_2} \leq \sum\limits^{n}_{i=1}  \|\bb{G}\bb{\Delta}^{k} \bb{s}^{(i)}\|_2 \leq  n \|\bb{G}\|_{2 \leftarrow 1} \lambda^k
    \end{align*}
    and leads to the desired result:
    \begin{align*}
         \|\sum\limits^{\infty}_{k=0} \bb{G}\bb{\Delta}^{k}\|_{\mathcal{H}_2} \leq  \sum\limits^{\infty}_{k=0} n \|\bb{G}\|_{2 \leftarrow 1} \lambda^k \leq \|\bb{G}\|_{2 \leftarrow 1}\frac{n}{1-\lambda}.
    \end{align*}
\end{proof}
\vspace{-5mm}
\lemref{lem:robust relax} gives an upper bound on the objective function of problem \eqref{eq: prob st SLS}. Hence, problem \eqref{eq: prob st SLS} can now be written as:
\begin{equation}\label{eq: prob st SLS relax}
    \begin{aligned}
    &\underset{\mathbf{\Phi}_{x}, \mathbf{\Phi}_{u}, \lambda}{\text{min}} & &\mathbb{E}_{\Vin_t \sim f}\left\Vert\begin{bmatrix}\mathbf{Q} & \mathbf{R}\end{bmatrix} \begin{bmatrix}\mathbf{\Phi}_{x}\\\mathbf{\Phi}_{u}\end{bmatrix}\right\Vert_{2 \leftarrow 1}\frac{n}{1-\lambda}\\
    & \text{s.t.} &  &\begin{aligned}
              &Z^T_{AB}\left[\begin{array}{c}\mathbf{\Phi}_{x}\\\mathbf{\Phi}_{u}\end{array}\right] = \mathbf{\Delta},\ \ \left\Vert{\mathbf{\Delta}}\right\Vert_{\mathcal{L}_1}<\lambda,\\
              &[\phi^{t}_x]_{i} = [[\phi^{t}_x]_{i}]_{\Vout_t(i)},\ [\phi^{t}_u]_{i} = [[\phi^{t}_u]_{i}]_{\Vout_t(i)}\\
              & \forall\ t=0,1,\ldots
              \end{aligned}
    \end{aligned}
\end{equation}

Problem \eqref{eq: prob st SLS relax} is a quasi-convex problem. Moreover, the $\mathcal{L}_1$-norm constraint can be decomposed in a convenient manner that facilitates distributed computation. In fact, applying our notational convention \eqref{eq:notat1}-\eqref{eq:notat3}, on $\bb{\Delta}$, we get the relation
\begin{align}\label{eq:decomprob}
    [Z^T_{AB}]\begin{bmatrix}[\phi^{t}_x]_{i} \\ [\phi^{t}_u]_{i}\end{bmatrix} = [\delta^{t}]_{i},
\end{align}
where $\Deltatkij{t}{:}{:}{i} := [\delta^{t}]_{i}$. Moreover, it can be verified that the term $\|\bb{\Delta}\|_{{\mathcal{L}_1}}$ can be equivalently written as
\begin{align}\label{eq:l1decomp}
    \|\bb{\Delta}\|_{{\mathcal{L}_1}} = \max\limits_{t,i} \|[\delta^{t}]_{i}\|_1,
\end{align} 
which allows to write the robustness constraint \eqref{eq: prob st SLS relax} as
\begin{align*}
    \max\limits_{t,i} \|[\delta^{t}]_{i}\|_1 \leq \lambda.
\end{align*} 
Recalling, that the $\|.\|_{2\leftarrow 1}$ norm also decomposes w.r.t. $[\phi^{t}_x]_{i}$ and $[\phi^{t}_u]_{i}$, we can combine equation \eqref{eq: G21},\eqref{eq:decomprob} and optimization \eqref{eq: prob st SLS relax}, to compute a relaxation of the robust optimal control problem \eqref{eq: prob st SLS} that decomposes. Furthermore, this allows for distributed computation, except for a global optimization over $\lambda$ and $c$, which can be solved via bisection in a distributed way through a consensus algorithm. In summary, every subsystem $i$ will solve the following optimization:
\begin{equation}\label{eq: SLS robust distr}
    \begin{aligned}
    &\underset{[\phi^{t}_x]_{i}, [\phi^{t}_u]_{i}, \lambda, c}{\text{min}} & &c\\
    & \text{s.t.} &  &\begin{aligned}
              & \mathbb{E}_{\Vin_t \sim f}\sqrt{\left\Vert\begin{bmatrix}[\mathbf{Q}]_i & [\mathbf{R}]_i\end{bmatrix} \begin{bmatrix}[\phi^{t}_x]_{i}\\ [\phi^{t}_u]_{i}\end{bmatrix}\right\Vert^2_{F}}\frac{n}{1-\lambda} \leq c\\
              &\left\Vert{[Z^T_{AB}]\begin{bmatrix}[\phi^{t}_x]_{i} \\ [\phi^{t}_u]_{i}\end{bmatrix}}\right\Vert_{1}<\lambda,\\
              &[\phi^{t}_x]_{i} = [[\phi^{t}_x]_{i}]_{\Vout_t(i)},\ [\phi^{t}_u]_{i} = [[\phi^{t}_u]_{i}]_{\Vout_t(i)}\\
              & \forall\ t=0,1,\ldots
              \end{aligned}
    \end{aligned}
\end{equation}

where, $[\mathbf{Q}]_i$ and $[\mathbf{R}]_i$ are the corresponding local blocks of subsystem $i$ for $Z_{AB}$, $\bb{Q}$ and $\bb{R}$ respectively.
\vspace{-1mm}
\subsection{Stability guarantees for a time-varying implementation}
\vspace{-3mm}
It is important to note that the controller synthesis problem \eqref{eq: SLS robust distr} corresponds to a LTI controller. However, the actual closed loop is time-varying due to the time-varying communication dropouts that can occur in the network. 
In the following lemma we show that column-wise switches\footnote{A column-wise switch means that we can change one or more columns of the transfer matrices ${\mathbf{\Phi}}_x$ and ${\mathbf{\Phi}}_u$.} between any number of $\mathcal{L}_1$ robust controllers do not render the closed-loop unstable. 
\\
\begin{lemma}
\label{thm: stability}
Consider the controller \eqref{eq:implementation}, where $\{[\phi^t_x]_i, [\phi^t_u]_i\} \subset \{\{\psi^{s}_{x},\psi^{s}_{u}\}\}^{M}_{s=1}$ where each $\{\psi^{s}_{x},\psi^{s}_{u}\}$ are partial system responses that satisfy
\begin{align*}
    \left\Vert{[Z^T_{AB}]\begin{bmatrix}[\psi^{s}_x] \\ [\psi^{s}_u]\end{bmatrix}}\right\Vert_{1}<1,
\end{align*}
then the closed loop system is stable. 
\end{lemma}{}
\begin{proof}
This result follows directly from the relationship \eqref{eq:l1decomp}. Recall that $\bb{\Delta}$ can be written as 
\begin{align*}
    \|\bb{\Delta}\|_{{\mathcal{L}_1}} = \max\limits_{t,i} \|[\delta^{t}]_{i}\|_1,
\end{align*} 
but since $\{[\phi^t_x]_i, [\phi^t_u]_i\} \in \{\{\psi^{s}_{x},\psi^{s}_{u}\}\}^{M}_{s=1}$, we have
\begin{align*}
    \|\bb{\Delta}\|_{{\mathcal{L}_1}} = \max\limits_{t,i} \|[\delta^{t}]_{i}\|_1 = \max\limits_{s} \left\Vert{[Z^T_{AB}]\begin{bmatrix}[\psi^{s}_x] \\ [\psi^{s}_u]\end{bmatrix}}\right\Vert_{1} < 1,
\end{align*}
which proves stability per small gain theorem.

\end{proof}{}
\vspace{-3mm}
Notice that this lemma is only valid under the assumption that we have $\mathcal{L}_1$ robustness, i.e. $\|\Delta\|_{\mathcal{L}_1}<1$, and given that the combinations are done column-wise. If this were not the case, an argument similar to the ones presented in \citep{ACCho2019} would need to be made. 

\section{Offline robustness to dropouts}
\vspace{-3mm}
In this section, we introduce a strategy to make the controller robust to communication dropouts. We take advantage of the formulation \eqref{eq: SLS robust distr} and \lemref{thm: stability} and we propose an offline controller synthesis that accounts for communication dropouts, making the resulting controller intrinsically robust to communication dropouts. Since robustness is guaranteed by the synthesis process, we can implement this controller offline. The resulting controller is a very low-cost controller robust to package losses. 
\vspace{-1mm}
\subsection{Offline controller synthesis}
\vspace{-3mm}
Let us start by considering that the general problem \eqref{eq: prob st SLS} with time invariant $\mathbf{\Phi}_x$ and $\mathbf{\Phi}_u$. We need to guarantee that the resulting control will be robust to all the dropouts suffered by the communication network. To do that, we need to consider all the different sparsity patterns induced by all the possible dropouts, and guarantee satisfaction of the constraints for all of them. According to the definition, $\Dout_t$ is associated with a given dropout at a certain time $t$. From equations \eqref{eq:sparsity}, we have that each $\Dout_t$ induces a sparsity pattern in $\mathbf{\Phi}_x$ and $\mathbf{\Phi}_u$. We can characterize all the dropouts scenarios generated by the probability mass function $f$, under which the controller is required to be robust as the set of sets $\mathscr{D}$, so $\Dout_t\in\mathscr{D}$ for each of the dropouts. Then, \eqref{eq: prob st SLS} can be written as:
\begin{equation}\label{eq: offline}
    \begin{aligned}
    &\underset{\mathbf{\Phi}_{x}, \mathbf{\Phi}_{u}}{\text{min}} & &\sum\limits_{\mathcal{S} \in \mathscr{D}} f(\mathcal{S})\left\Vert\begin{bmatrix}\mathbf{Q} & \mathbf{R}\end{bmatrix} \begin{bmatrix}[\mathbf{\Phi}_{x}]_{\mathcal{S}} \\ [\mathbf{\Phi}_{u}]_{\mathcal{S}} \end{bmatrix}(I+\mathbf{\Delta})^{-1}\right\Vert_{\mathcal{H}_2}\\
    & s.t. &  &\begin{aligned}
              &\left\|Z^T_{AB}\begin{bmatrix}[\mathbf{\Phi}_{x}]_{\mathcal{S}} \\ [\mathbf{\Phi}_{u}]_{\mathcal{S}} \end{bmatrix}\right\|_{\mathcal{L}_1} <1, \forall\ \mathcal{S} \in\mathscr{D},\\
              &[[\bb{\Phi}_x]_{i}]_{V/\Vmaxout} = 0,\ 
               [[\bb{\Phi}_u]_{i}]_{V/\Vmaxout} = 0.
              \end{aligned}
    \end{aligned}
\end{equation}

Notice that the tractability of \eqref{eq: offline} depends on the size of $\mathscr{D}$ and $\Vmaxin$ and $\Vminout$. Even thought this can lead to a combinatorial explosion, since the problem can be solved in a distributed manner via equation \eqref{eq: SLS robust distr} and the size of $\Vmaxin$ and $\Vminout$ is much smaller than the size of the network in structured systems, this represents a feasible approach in practice. 

 \vspace{-1mm}
\subsection{Offline controller implementation}
\vspace{-3mm}
The controller implementation is distributed and can be described by equations \eqref{eq:implementation}, where the maps $\Phi_x$ and $\Phi_u$ are obtained from the distributed synthesis \eqref{eq: SLS robust distr} with the additional robustness constraints as discussed in the previous subsection.

Notice that although the controller was restricted to be LTI, the implementation of such a controller will be time varying due to the time varying nature of the dropouts. This controller is however guaranteed to be stable by virtue of \lemref{thm: stability}.

\section{Online robustness to dropouts}
\vspace{-3mm}
Here we introduce a different strategy to make the controller robust to communication dropouts. Once again, we take advantage of the formulation \eqref{eq: SLS robust distr} and \lemref{thm: stability} and we propose an online controller synthesis that is able to sense communication dropouts instantaneously. In this design performance is not hurt by the robustness, at the expense that the controller needs an online implementation and is therefore time-varying.
\vspace{-1mm}
\subsection{Online controller synthesis}
\vspace{-3mm}
We start by considering the general problem \eqref{eq: prob st SLS} were the controller is allowed to be time varying. The goal is for the controller to perform the optimal action given the current dropout. In order to do so, it is important to recall that the controller is able to sense instantaneously the dropout experienced. This is a reasonable assumption assuming a handshake protocol is in place and communication occurs at a local scale. Hence, the optimal strategy would be to use the optimal $\Phi_x$ and $\Phi_u$ for each dropout when it occurs, therefore switching between optimal controllers based on the dropout experienced. Under this premise, the synthesis reduces to solving offline optimization \eqref{eq: prob st SLS} for all dropout cases, which can be rewritten as:
\begin{equation}\label{eq: online}
    \begin{aligned}
    &\underset{\mathbf{\Phi}^{\mathcal{S}}_{x}, \mathbf{\Phi}^{\SetS}_{u}}{\text{min}} & &\left\Vert\begin{bmatrix}\mathbf{Q} & \mathbf{R}\end{bmatrix} \begin{bmatrix}\mathbf{\Phi}_{x}\\\mathbf{\Phi}_{u}\end{bmatrix}(I+\mathbf{\Delta})^{-1}\right\Vert_{\mathcal{H}_2}\\
    & s.t. &  &\begin{aligned}
              &\left\|Z^T_{AB}\left[\begin{array}{c}\mathbf{\Phi}^{\SetS}_{x}\\\mathbf{\Phi}^{\SetS}_{u}\end{array}\right]\right\|_{\mathcal{L}_1}<1,
              \end{aligned}
    \end{aligned}
\end{equation}
for each dropout pattern $\SetS\in\mathscr{D}$. The number of optimizations to solve depends on the size of $\mathscr{D}$ and the different sparsity patterns that they generate in $\bb{\Phi}_x$ and $\bb{\Phi}_u$. Notice that problem \eqref{eq: online} can be solved offline in a distributed manner via \eqref{eq: SLS robust distr}. 
\vspace{-1mm}
\subsection{Online controller implementation}
\vspace{-3mm}
Once the controllers have been synthesized offline, the online implementation easily follows. If at every time $t$ dropout $\Dout_t(i)$ is sensed, subsystem $i$ implements the $[\phi_x^{t+1}]_i$ and $[\phi_u^{t+1}]_i$ that were synthesized according to the sparsity pattern induced by $\Dout_t(i)$. The corresponding $[\phi_x^{t+1}]_i$ and $[\phi_u^{t+1}]_i$ are implemented according to \eqref{eq:implementation}. This repeats for each $t$.

\section{Simulation Experiments}
\vspace{-3mm}
In this section, we present simulation experiments of the two strategies introduced in this paper to deal with dropouts and we compare their performance for different dropout scenarios. To perform these experiments we choose the following dynamical system consisting of $10$ nodes:
$$[x]_i(t+1) = 1.2\sum_{j=i-2}^{i+2} \alpha_{ij}[x]_{j}(t) + 1.2[u]_i(t),$$
where $\alpha_{ij} = .4$ for $j=i\pm1$ and $\alpha_{ij} = .2$ otherwise. For $i=1,10$ $\alpha_{ii}=.6$. Negative $j$ are not considered. The communication is described by the following adjacency matrix:
\begin{equation*} 
    G= \left[\begin{array}{cccccc}1 & 1 & 0 & 0 & 0 & \dots \\ 1 & 1 & 1 & 0 & 0& \dots\\ 0 & 1 & 1 & 1 & 0 & \dots\\ \vdots &  & \ddots &  \ddots &\ddots &\ddots \end{array}\right]^d,
\end{equation*}
and we consider $d$ to be the dropout parameter. If no dropouts are present $d=5$. A dropout changes the value of $d$, so $d\in\mathcal{D}=\left\{2,3,4,5\right\}$. The probability distribution of $d$ over $\mathcal{D}$ is uniform. We define one $[d]_i\in\mathcal{D}$ for each subsystem $i$, which represents the sparsity induced in the $i^{th}$ column of $(G)^{[d_i]}$. 

The cost at each time is computed as $C(t)=x(t)^{\mathsf{T}}x(t)+u(t)^{\mathsf{T}}u(t)$, and the total cost for each simulation is computed as $C=\sum_{t=0}^{T} C(t)$ for $T=100$. The simulations are run comparing the two strategies -- offline and online -- subject to the same random noise and the same dropout scenario. An illustrative example of the state, input and communication topology is introduced in \figref{fig: state}.

\begin{figure}[h]
    \includegraphics[width=8.7cm]{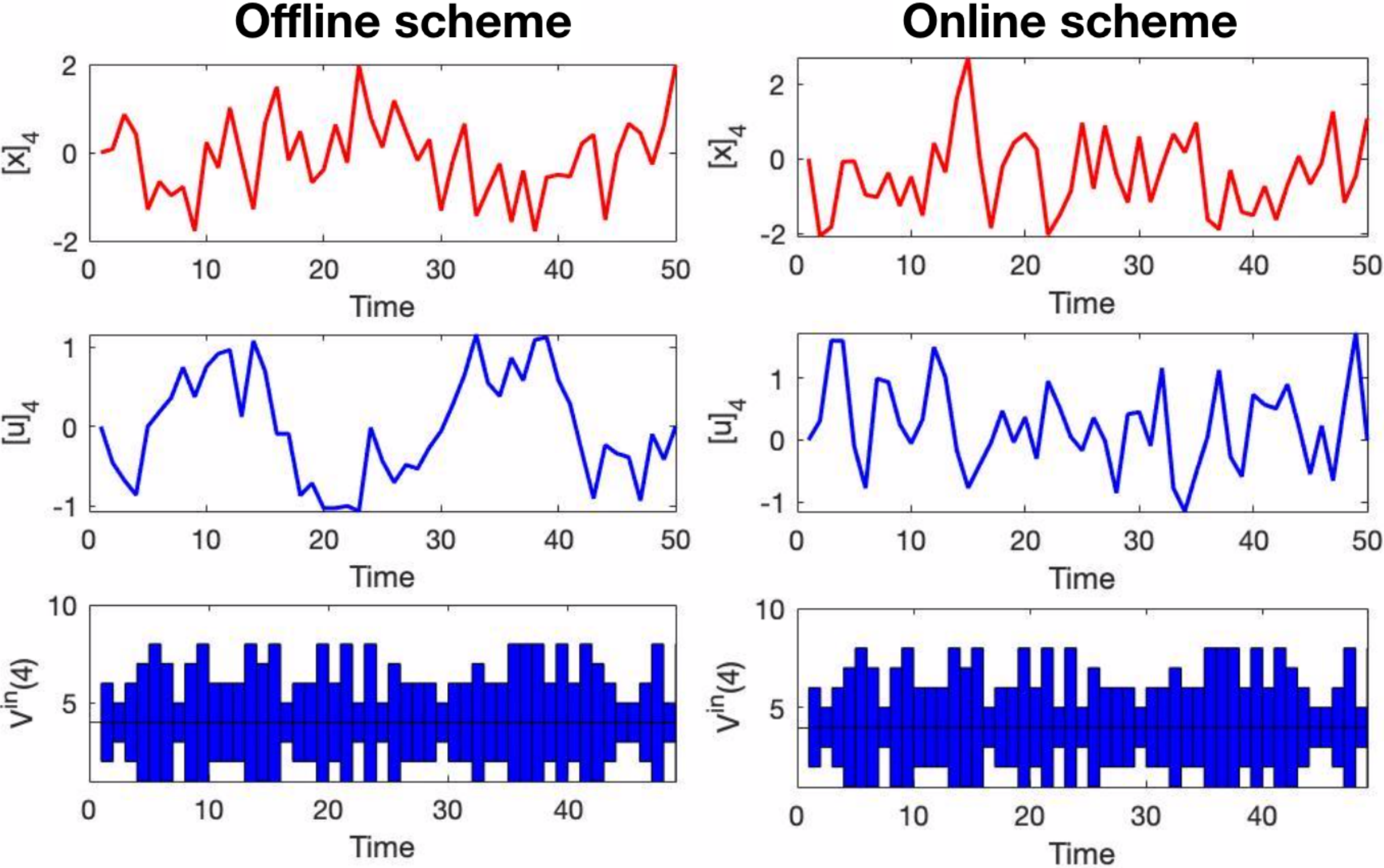}
    \caption{On the left: state, input and communication topology for subsystem $4$ using the offline strategy. On the right: identical representation for the online scheme.}
    \label{fig: state}
\end{figure}

Further, we compare the moving average defined as $M=\frac{1}{T}\sum_{t=1}^{T}\mathbb{E}[C(t)]$, where $\mathbb{E}[C(t)]$ is by averaging the cost under $10$ different Gaussian noise processes. We do so for different dropout scenarios. As illustrated by \figref{fig: average}, the online strategy performs better than the offline one. However, the gap is not dramatic, so the offline strategy is more cost efficient since it does not require an online implementation. Simulations suggest that both of these strategies are robust to communication dropouts while providing good performance, and the choice of using one versus the other is context dependent. 

\begin{figure}[h]
    \centering
    \includegraphics[width=9cm]{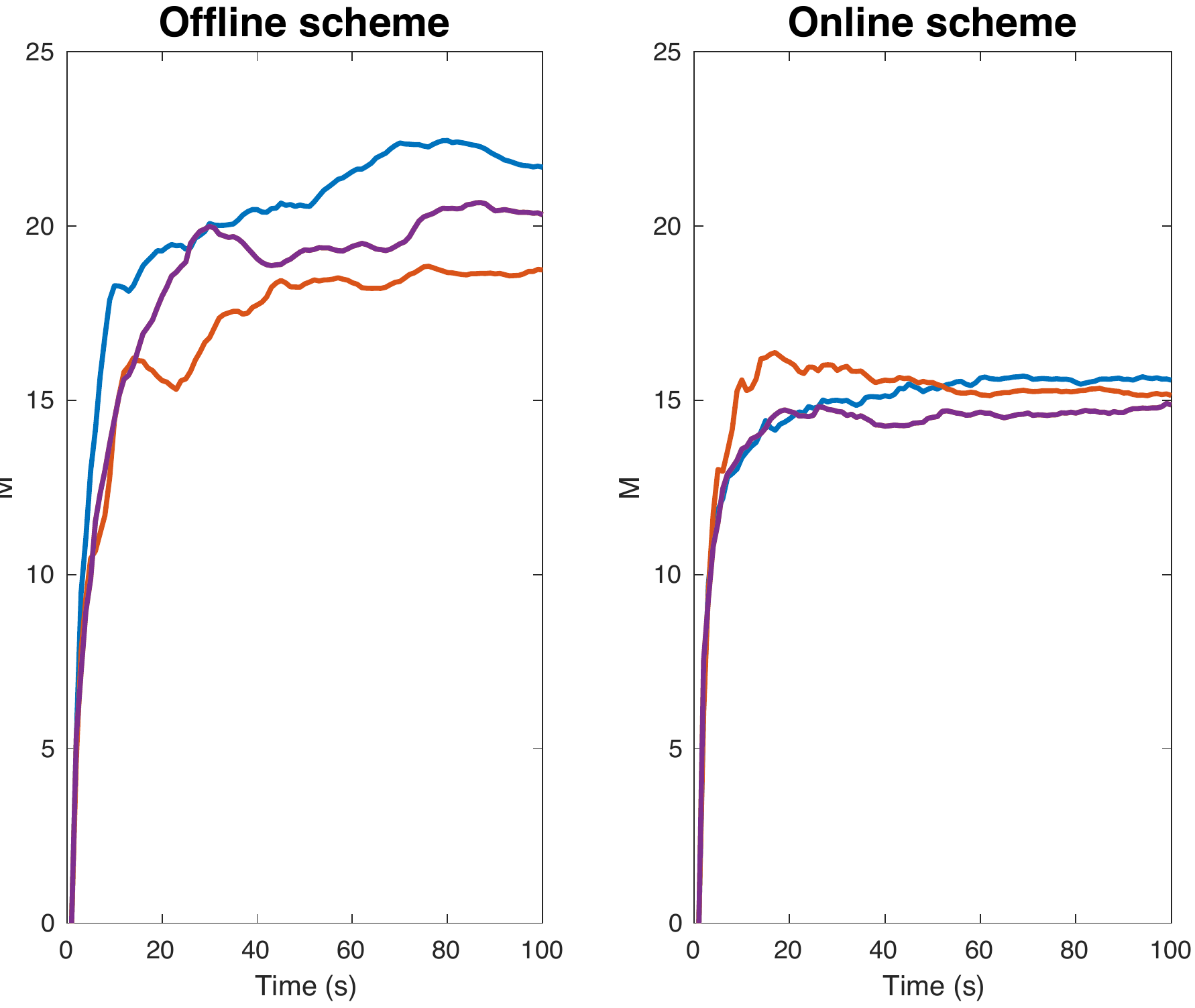}
    \caption{On the left: moving average LQR cost computed over $10$ disturbance processes for three different dropout scenarios for the offline scheme. On the right: identical representation for the online scheme.}
    \label{fig: average}
\end{figure}

\section{Conclusion}
\vspace{-3mm}
We presented two different strategies to deal with information packages dropouts in the communication network of distributed controllers by leveraging the SLS framework. The first strategy consists of an offline synthesis of a SLS controller constrained such that the resulting controller is inherently robust to communication dropouts. This controller can then be implemented in an offline fashion with the guarantee that it will be robust touts. The second strategy consists of a synthesis of a collection of SLS controllers, each of them being optimal for a certain sparsity pattern. The implementation of the controller is carried out online, and each agent is able to choose a realization at each time step from the collection of synthesized controllers based on the current communication topology originated by the dropouts, which it can sense instantaneously. Notice that, although the controllers are synthesized as LTI, the controller implemented are time varying due to the time-varying nature of the communication network. We show in \lemref{thm: stability} that the controllers implemented are internally stabilizing. We also provide a relaxation to the robust version of SLS that allows for a distributed computation.

This work represents a first step into the use of SLS as a tool to tackle communication problems in a networked control setting. Here we only discuss the extensions for the linear time-varying case, but remark that it is an interesting open problem whether the presented techniques can be extended to a even a broader setting, using the SLS approach for nonlinear systems introduced in \citep{ho2020system}. As for future work, it remains an open question how to tackle communication dropouts in the case where delay is present. Furthermore, we plan to exploit the connections with distributed and localized model predictive control \citep{amo_alonso_distributed_2019}, since communication in this scheme is key and drop of packages is likely in real-life applications. Another application could be the coalitional control framework~ \citep{fele_coalitional_2018}, where cooperating local controllers are clustered in disjoint groups with intermittent communication to promote a trade-off between closed-loop performance and coordination overheads. 

\appendix
 \section{Supplementary proofs}      
\vspace{-3mm}

\begin{lemma}\label{lem:Neuma}
Let $\bb{\Delta}$ be a linear bounded operator on $\mathcal{L}_1$ and assume $\|\bb{\Delta}\|_{\mathcal{L}_1} <1$, then $(\bb{I} - \bb{\Delta})^{-1}$ exists, can be written equivalently as
$$(\bb{I} - \bb{\Delta})^{-1} = \sum\limits^{\infty}_{k=0} \bb{\Delta}^{k}$$
and is bounded by $\|(\bb{I} - \bb{\Delta})^{-1}\|_{\mathcal{L}_1} \leq \frac{1}{1-\lambda}$
where $\bb{\Delta}^{k} = \underbrace{\bb{\Delta} \circ \bb{\Delta} \circ \dots \circ \bb{\Delta} }_{k}$.
\end{lemma}

\textbf{Proof of Lemma \lemref{lem:G21}}
\vspace{-3mm}
\begin{proof} 
Decompose $\bb{w}$ into $$\bb{w} = \sum^{n}_{i=1}\sum^{\infty}_{t=0} \bb{s}^{(i,t)} w_{i,t},$$ where $\bb{s}^{(i,t)}$ are dirac sequences, i.e. $s^{(i,t)}_{i,t} = 1$ and $s^{(i,t)}_{j,k} = 0$ for all other times $k$ and vector entries $j$. Now, due to linearity and triangle inequality we have
\begin{align*}
    \left\|\bb{\Phi}\bb{w}\right\|_2 = \|\sum^{n}_{i=1}\sum^{\infty}_{t=0} \bb{\Phi}\bb{s}^{(i,t)} w_{i,t}\|_2 \leq \sum^{n}_{i=1}\sum^{\infty}_{t=0} \|\bb{\Phi}\bb{s}^{(i)}\|_2 |w_{i,t}|,
\end{align*}
where we dropped the $t$ super-index due to time-invariance of $\bb{G}$. It follows the inequality
\begin{align*}
    \left\|\bb{\Phi}\bb{w}\right\|_2 &\leq \max_i \sqrt{\|[\phi^{0}]_{i}\|^{2}_{F}} \sum^{n}_{i=1}\sum^{\infty}_{t=0} |w_{i,t}|\\
    & \leq \|\bb{\Phi}\|_{2 \leftarrow 1} \|\bb{w}\|_{1}.
\end{align*}
The result follows, because the left-hand bound can always be achieved with an appropriately chosen dirac sequence $\bb{w}$.
\end{proof}

\bibliography{references}         

\end{document}